\documentclass[11pt, reqno]{amsart}

\usepackage{amsmath,amssymb,amsthm, epsfig}
\usepackage{hyperref}
\usepackage{amsmath}
\usepackage{amssymb}
\usepackage{color}
\usepackage{ulem}
\usepackage{epsfig}
\usepackage[mathscr]{eucal}

\usepackage[latin1]{inputenc}

\newtheorem{theorem}{Theorem}
\newtheorem{definition}{Definition}

\newtheorem{lemma}{Lemma}
\newtheorem{proposition}{Proposition}
\newtheorem{corollary}{Corollary}
\newtheorem{remark}{Remark}

\date{}
\numberwithin{equation}{section} \numberwithin{theorem}{section}
\numberwithin{lemma}{section} \numberwithin{corollary}{section}
\numberwithin{remark}{section} \numberwithin{proposition}{section}
\numberwithin{definition}{section}

\newcommand{\loc}{\operatorname{loc}}

\def \R {\mathbb{R}}
\def \supp {\mathrm{supp } }

\def \loc {\mathrm{loc}}

\newcommand{\pe}{E_{\varepsilon}}

\newcommand{\intav}[1]{\mathchoice {\mathop{\vrule width 6pt height 3 pt depth  -2.5pt
\kern -8pt \intop}\nolimits_{\kern -6pt#1}} {\mathop{\vrule width
5pt height 3  pt depth -2.6pt \kern -6pt \intop}\nolimits_{#1}}
{\mathop{\vrule width 5pt height 3 pt depth -2.6pt \kern -6pt
\intop}\nolimits_{#1}} {\mathop{\vrule width 5pt height 3 pt depth
-2.6pt \kern -6pt \intop}\nolimits_{#1}}}

\begin{document}

\title[Singularly perturbed fully nonlinear parabolic problems]{Singularly perturbed fully nonlinear parabolic problems and their asymptotic free boundaries}

\author[G. C. Ricarte]{Gleydson C. Ricarte}
\address{Universidade Federal do Cear\'{a} - UFC, Department of Mathematics, Fortaleza - CE, Brazil - 60455-760.}
\email{ricarte@mat.ufc.br}

\author[R. Teymurazyan]{Rafayel Teymurazyan}
\address{CMUC, Department of Mathematics, University of Coimbra, 3001-501 Coimbra, Portugal.}
\email{rafayel@mat.uc.pt}

\author[J.M. Urbano]{Jos\'e Miguel Urbano}
\address{CMUC, Department of Mathematics, University of Coimbra, 3001-501 Coimbra, Portugal.}
\email{jmurb@mat.uc.pt}

\begin{abstract}
We study fully nonlinear singularly perturbed parabolic equations and their limits. We show that solutions are uniformly Lipschitz continuous in space and H\"{o}lder continuous in time. For the limiting free boundary problem, we analyse the behaviour of solutions near the free boundary. We show, in particular, that, at each time level, the free boundary is a porous set and, consequently, is of Lebesgue measure zero. For rotationally invariant operators, we also derive the limiting free boundary condition.
\bigskip


\bigskip

\noindent \textbf{AMS Subject Classifications MSC 2010:} 35K55, 35D40, 35B65, 35R35.

\end{abstract}

\maketitle

\section{Introduction}

In this paper we study the following singular perturbation problem for a fully nonlinear parabolic equation
\begin{equation}\label{Equation Pe} \tag{$\pe$}
\left\{
\begin{array}{rclcl}
F(x,t,D^2u_{\varepsilon}) - \partial_t u_{\varepsilon} & = & \beta_{\varepsilon}(u_{\varepsilon}) + f_{\varepsilon} & \mbox{in} & \Omega_T \\
u_{\varepsilon} & = & \varphi & \mbox{on} & \partial_{p} \Omega_T,
\end{array}
\right.
\end{equation}
where $F(x,t,M)$ is a fully nonlinear uniformly elliptic operator, the Dirichlet data $\varphi$ is nonnegative and the singularly perturbed potential $\beta_{\varepsilon}(\cdot)$ is a suitable approximation of a multiple of the Dirac mass $\delta_0$. The problem appears, for example, in combustion theory and describes the propagation of curved, premixed deflagration flames.  It is derived (cf. \cite{Bu}) in the framework of the theory of equidiffusional premixed flames, analysed in the relevant limit of right activation energy for Lewis number equal to one, and the unknown $u^{\varepsilon}$ represents the normalised temperature of the mixture.

The study of the limit as $\varepsilon \to 0$ in (\ref{Equation Pe}) (the high activation energy analysis) leads to a free boundary problem, and often provides an alternative way of approaching questions related to the existence and the regularity of solutions and the free boundary.  For example, the one-phase elliptic problem
\begin{equation}\label{liverpool}
\left\{
\begin{array}{rclcl}
\Delta u  &=&  0 & \mbox{in} & \{u >0\} \vspace*{0.2cm}\\
|\nabla u|& = & C & \mbox{on} & \partial \{u>0\},
\end{array}
\right.
\end{equation}
studied by Alt and Caffarelli in \cite{ALT}, can be approached by taking $\varepsilon \to 0$ in
$$
\Delta u_{\varepsilon} = \beta_{\varepsilon}(u_{\varepsilon}).
$$
In  \cite{ALT}, it is shown that any minimiser $u$ of the problem
$$
 \int_{\Omega} |\nabla v|^2 + \chi_{\{v >0\}} \rightarrow \textrm{min}
$$
is Lipschitz continuous and solves \eqref{liverpool} with a nonnegative Dirichlet boundary condition. Alt and Caffarelli also proved that the free boundary condition
holds in a weak sense, and that the free boundary $\partial \{u>0\}$ is a $C^{1,\alpha}$ surface except at a set of zero surface measure.

The idea of passing to the limit in a singular perturbation problem had been proposed in \cite{Zeldovich} but would only be treated rigorously in \cite{BCN}, in the one-phase case (that is, with $u \ge 0$), for general linear operators. The results in \cite{BCN} include the Lipschitz continuity of the limit, the fact that it solves the free boundary problem in a weak sense and some geometric measure properties of particular level sets. The topic would become the object of intense research and we highlight the contributions of \cite{Caf2,Caf3,Wol,LW98,WM,MT}, where, in particular, the two-phase problem (allowing $u$ to change sign) was treated. The parabolic case
$$
\Delta u_{\varepsilon} - \partial_t u_{\varepsilon} = \beta_{\varepsilon}(u_{\varepsilon})
$$
was studied in \cite{Caf4} for one phase and in \cite{Caf1, Caf2, Caf3} for the two-phase problem.

This alternative approach opens an avenue leading also to non-variational free boundary problems. Recently, the singular perturbation problem
$$
F(x, D^2 u_{\varepsilon}) = \beta_{\varepsilon}(u_{\varepsilon}),
$$
which is the elliptic counterpart of (\ref{Equation Pe}), was studied in \cite{RT11}; the authors obtain Lipschitz estimates and study the limiting free boundary problem. Our aim in this paper is to extend these results to the parabolic case. We consider a family of solutions of problem \eqref{Equation Pe} and show that, under suitable assumptions, the limit function $u$ is a solution to the free boundary problem
\begin{equation}\label{1.2}
\left\{
\begin{array}{rclcl}
F(x,t, D^2 u)-\partial_t u  &=&  f & \mbox{in} & \{u >0\}\\
u & = & \varphi & \mbox{on} & \partial_p \Omega_T\\
\end{array}
\right.
\end{equation}
where $f =\lim f_{\varepsilon}$. We do not impose a free boundary condition and thus the limiting problem is not understood as overdetermined.

Unlike the elliptic case (see, for example, \cite{RT11}), one can not apply the Harnack inequality in order to prove the (uniform) regularity of solutions. The reason is that  we can only compare functions on parabolic boundaries, not on the \textit{top} of a cylinder; we are thus unable to pass from one level to another. We overcome this difficulty by using a Bernstein type argument (see the proof of Proposition \ref{p4.1}). For the same reason, the study of the free boundary of the limiting problem requires a totally different approach: in the elliptic case, using a covering argument, one can prove the finiteness of the $(n-1)$-dimensional Hausdorff measure of the free boundary (see \cite{RT11}). In the parabolic case, what we are able to prove is that, at each time level, the $n$-dimensional Lebesgue measure of the free boundary is zero because it is porous. We prove this by obtaining a non-degeneracy result and by controlling the growth rate of the solution near the free boundary. For rotationally invariant operators, we also derive the limiting free boundary condition, which is the natural parabolic extension of the condition in the elliptic case.

The paper is organised as follows. We first prove the existence of solutions to \eqref{Equation Pe} using Perron's method. We also show in Section \ref{s3} that solutions are uniformly bounded (Theorem \ref{t3.2}). In Section \ref{s4}, using a Bernstein type argument, we obtain a uniform gradient estimate for solutions (Proposition \ref{p4.1}), which implies the uniform H\"older continuity in time with exponent $1/2$ (Proposition \ref{p4.2}), just as in the classical case of the heat equation. In Section \ref{s5}, we pass to the limit in \eqref{Equation Pe} as $\varepsilon\rightarrow0$. Invoking stability arguments, we show that the limit function is a solution of a free boundary problem (Theorem \ref{t5.1}). The regularity of the free boundary is then studied in Section \ref{s6}: we first prove the non-degeneracy of the solution of the limiting free boundary problem (Lemma \ref{l6.1}) and next establish the growth rate of the solution near the free boundary (Lemma \ref{l6.3}). These two results lead to the porosity of the free boundary at each time level (Theorem \ref{t6.1}). Finally, in Section \ref{s7} we derive the free boundary condition in the case of rotationally invariant operators (Theorem \ref{t7.1}).

\section{Mathematical set-up}\label{s2}

Given a bounded domain $\Omega \subset \mathbb{R}^n$, with a smooth boundary $\partial \Omega$, we define, for $T>0$, $\Omega_T = \Omega \times (0,T]$, its lateral boundary $\Sigma = \partial \Omega \times (0,T)$ and its parabolic boundary $\partial_p \Omega_T = \Sigma \cup (\Omega \times \{0\})$.

An operator $F : \Omega_T \times \mathbb{R} \times \textrm{Sym}(n) \rightarrow \mathbb{R}$ is uniformly elliptic if there exist two positive constants $\lambda \le \Lambda$ (the ellipticity constants) such that, for any $M \in \textrm{Sym}(n)$ and $(x,t) \in \Omega_T$,
\begin{equation} \label{2.1}
 \lambda \|P\| \le F(x,t, M+P) - F(x,t, M) \le \Lambda \|P\|,
\end{equation}
for every non-negative definite symmetric matrix $P$. Here, $\textrm{Sym}(n)$ is the space of real $n \times n$ symmetric matrices and $\|P\|$ equals the maximum eigenvalue of $P$.

We let $\mathcal{P}^{-}_{\lambda, \Lambda}$ and $\mathcal{P}^{+}_{\lambda, \Lambda}$ denote the minimal and maximal Pucci extremal operators corresponding to $\lambda, \Lambda$, that is, for $M \in \textrm{Sym}(n)$,
$$
	\mathcal{P}^{-}_{\lambda, \Lambda}(M) = \lambda \sum_{e_i >0} e_i + \Lambda \sum_{e_i <0} e_i \quad \textrm{and} \quad
	\mathcal{P}^{+}_{\lambda, \Lambda}(M) = \Lambda \sum_{e_i >0} e_i + \lambda \sum_{e_i <0} e_i,
$$
where $\left\{ e_i = e_i(M), 1 \le i \le n\right\}$ is the set of eigenvalues of $M$.   We recall also that
$$
\mathcal{P}^{-}_{\lambda, \Lambda}(M) = \inf_{A \in \mathcal{A}_{\lambda, \Lambda}} \textrm{tr}(AM) \quad \textrm{and} \quad  \mathcal{P}^{+}_{\lambda, \Lambda}(M) = \sup_{A \in \mathcal{A}_{\lambda, \Lambda}} \textrm{tr}(AM),
$$
where
$\mathcal{A}_{\lambda, \Lambda} = \left\{ A \in \textrm{Sym}(n) : \lambda |\xi|^2 \le A_{ij} \xi_i \xi_j \le \Lambda |\xi|^2, \, \forall \, \xi \in \mathbb{R}^n\right\}$.
Note that uniform ellipticity implies that, for $A, B \in \textrm{Sym}(n)$,
\begin{equation} \label{Pucci}
\mathcal{P}^{-}_{\frac{\lambda}{n}, \Lambda}(A-B) \le F(x,t,A)-F(x,t,B) \le \mathcal{P}^{+}_{\frac{\lambda}{n}, \Lambda}(A-B).
\end{equation}
Any operator $F$ which satisfies condition \eqref{2.1} will be referred to as a $(\lambda,\Lambda)$-elliptic operator.

We now define, following \cite{LW,Wang}, the notion of viscosity solution for a fully nonlinear parabolic equation.
\begin{definition}
A function $u \in C(\Omega_T)$ is a viscosity sub-solution (\mbox{resp.} super-solution) of
$$
  F(x,t, D^2 u)-\partial_t u = g(x,t,u) \quad \mbox{in} \quad \Omega_T
$$
if, whenever $\phi \in C^2(\Omega_T)$ and $u-\phi$ has a local maximum (\mbox{resp.} minimum) at $(x_0,t_0) \in \Omega_T$, there holds
$$
 F(x_0,t_0,D^2 \phi (x_0,t_0))- \partial_t \phi (x_0,t_0) \geq g(x_0,t_0, \phi(x_0,t_0)). \quad (\mbox{resp.} \leq )
$$
A function $u$ is a viscosity solution if it is both a viscosity sub-solution and a viscosity super-solution.
\end{definition}

We also define the class of functions, that will be useful in the sequel,
$$
\mathscr{S}(\lambda, \Lambda, f) := \overline{\mathscr{S}}(\lambda,\Lambda,f) \cap \underline{\mathscr{S}}(\lambda,\Lambda,f).
$$
where
\begin{eqnarray*}
\overline{\mathscr{S}}(\lambda,\Lambda,f) := \left\{ u \in C(\Omega_T) : \mathcal{P}^{-}(D^2u)-\partial_t u \le f \  \textrm{in} \  \Omega_T \right\}\\
\underline{\mathscr{S}}(\lambda,\Lambda,f) := \left\{ u \in C(\Omega_T) : \mathcal{P}^{+}(D^2u)-\partial_t u \ge f \ \textrm{in} \ \Omega_T  \right\},
\end{eqnarray*}
the inequalities taken in the viscosity sense.

We need to clarify what is a Lipschitz function defined in a space-time domain.

\begin{definition}
Let $\mathcal{D}\subset \mathbb{R}^{n} \times  \mathbb{R}$. We say that $v\in\textrm{Lip}_{\loc}(1,1/2)(\mathcal{D})$ if, for every compact $K \Subset \mathcal{D}$, there exists a constant $C=C(K)$ such that
$$
|v(x,t) - v(y,s)| \le C \left(|x-y| + |t-s|^{\frac{1}{2}}\right),
$$
for every $(x,t), (y,s) \in K$.  If the constant $C$ does not depend on the set $K$ we say $v \in \textrm{Lip}(1,1/2)(\mathcal{D})$.
\end{definition}
We also define the $\textrm{Lip}(1,1/2)(\mathcal{D})$ seminorm in $\mathcal{D}$
$$
[v]_{\textrm{Lip}(1,1/2)(\mathcal{D})} := \sup_{(x,t),(y,s) \in \mathcal{D}} \frac{|v(x,t)-v(y,s)|}{|x-y| + |t-s|^{1/2}}
$$
and the $\textrm{Lip}(1,1/2)(\mathcal{D})$ norm in $\mathcal{D}$
$$
\|v\|_{\textrm{Lip}(1,1/2)(\mathcal{D})} := \|v\|_{L^{\infty}(\mathcal{D})} + [v]_{\textrm{Lip}(1,1/2)(\mathcal{D})}.
$$

For future reference and further clarity, we gather next the set of assumptions concerning the data in \eqref{Equation Pe}.

\bigskip

\noindent {\bf Assumptions on the data for \eqref{Equation Pe}.}

\bigskip

\begin{description}

\item[(A1)] $F=F(x,t,M)$ is uniformly elliptic, concave and of class $C^{1,\alpha}$ in $M$ and of class $C_{\mathrm{loc}}^{1,\alpha}$ in $(x,t)$, for some $\alpha >0$, and $F(\cdot,\cdot,0)=0$. 

\item[(A2)] The singular reaction term $\beta_{\varepsilon} : \mathbb{R}_{+} \rightarrow \mathbb{R}_{+}$ satisfies
$$
	0 \le \beta_{\varepsilon}(s) \le \frac{1}{\varepsilon} \chi_{(0,\varepsilon)}(s), \quad \forall \, s \in \mathbb{R}_{+}.
$$
For example, it can be built as an approximation of unity
$$
	\beta_{\varepsilon}( s) := \frac{1}{\varepsilon} \beta \left(\frac{s}{\varepsilon}\right),
$$
where  $\beta$ is a nonnegative smooth real function with $\supp ~\beta = [0,1]$, such that
$$
\|\beta\|_\infty\le 1 \quad \textrm{and} \quad \int_{\mathbb{R}} \beta (s) \, ds < \infty.
$$
Such a sequence of potentials converges, in the distributional sense, to $\int \beta$ times the Dirac measure $\delta_0$.

\item[(A3)] $f_{\varepsilon} (x,t) \in C^{1,\alpha}(\overline{{\Omega}_T})$, is non-increasing in $t$ and satisfies
$$0 < c_0 \le f_\varepsilon(x,t) \le c_1 < \infty  \quad \textrm{in} \quad \Omega_T$$
and
$$\|\nabla f_{\varepsilon}\|_{\infty} \leq C.$$

\item[(A4)] The Dirichlet data $0\leq \varphi (x,t) \in C^{1,\alpha}(\partial_p\Omega_T)$, is non-decreasing in $t$ and satisfies $\varphi(x,0)=0$.

\end{description}

\bigskip
Finally, we introduce some further notation.
\bigskip

\noindent {\bf Notation.} For $x_0 \in \mathbb{R}^n$, $t_0 \in \mathbb{R}$ and $\tau >0$, we denote
\begin{eqnarray*}
 B_{\tau}(x_0) &:=& \left\{x \in \mathbb{R}^{n} : |x-x_0| < \tau \right\}, \\
 Q_{\tau}(x_0,t_0) &:=& B_{\tau}(x_0) \times (t_0-\tau^2, t_0 + \tau^2),\\
 Q^{-}_{\tau}(x_0,t_0) &:=& B_{\tau}(x_0) \times (t_0-\tau^2, t_0],
\end{eqnarray*}
and, for a set $K \subset \mathbb{R}^{n+1}$ and $\tau >0$,
$$\mathcal{N}_{\tau}(K) := \bigcup_{(x_0,t_0) \in K} Q_{\tau}(x_0,t_0) \quad \mathrm{and} \quad \mathcal{N}^{-}_{\tau}(K) := \bigcup_{(x_0,t_0) \in K} Q^{-}_{\tau}(x_0,t_0).$$

\section{Existence of viscosity solutions}\label{s3}

Our first goal is to show that \eqref{Equation Pe} has at least one viscosity solution. Because of the lack of monotonicity of equation \eqref{Equation Pe} with respect to the variable $u$, the classical Perron's method can not be applied directly. The following result is a suitable adaptation, stated in a more general form, since we feel it may be of independent interest.

\begin{theorem}\label{t3.1}
Let $F$ satisfy $(A1)$, $g\in C^{0,1}(\mathbb{R})\cap L^{\infty}(\mathbb{R})$, $f \in C(\Omega_T)$ and $\varphi\in C(\partial_p \Omega_T)$. If $u_{\star}$, $u^{\star}$ are, correspondingly, a viscosity sub-solution and a viscosity super-solution of
\begin{equation}\label{3.1}
F(x,t,D^2 u)-\partial_tu= g(u)+f \quad \textrm{in} \quad \Omega_T,
\end{equation}
with $u_{\star}=u^{\star}=\varphi$ on $\partial_p \Omega_T$, then 
$$u:=\displaystyle\inf_{v\in\mathcal{S}}v$$ 
is a viscosity solution of \eqref{3.1}, where 
$$\mathcal{S}:=\{v\in C(\overline{{\Omega}_T});\, u_\star\le v \le u^{\star} \textrm{ and $v$ is a super-solution of \eqref{3.1}}\}.$$
\end{theorem}

\begin{proof}
Let $\mu>0$ be such that $|g'|< \mu/2$ and let $h(z):=\mu z - g(z)$, which is then increasing. For $\psi \in C^{0,1}( \overline{{\Omega}_T})$ we define the following (uniformly elliptic) operator
$$
G_{\psi}[u]:= G_{\psi}(x,t,u, D^2u):= F(x,t,D^2u) - \mu u-f + \psi.
$$
Next, set $u_{0}:=u_{\star}$ and let $u_{k+1}$ be a solution of
\begin{equation}\label{3.2}
     \left \{
        \begin{array}{rclcl}
            G_{\psi_{k}}[u] -\partial_tu&=& 0 & \textrm{ in }& \Omega_T \\
            u&=& \varphi\ & \textrm{ on } & \partial_p \Omega_T,
        \end{array}
    \right.
\end{equation}
where  $\psi_{k}=h(u_{k})$. The existence of a solution to \eqref{3.2} is assured by the classical Perron's method (see \cite{CIL92,I87}), since $G_{\psi}(x,t,r,M)$ is now non-increasing in $r$. We claim that
\begin{equation}\label{3.3}
 u_{\star}=u_{0}\leq u_{1}\leq \dots\leq u_{k}\leq u_{k+1}\leq\dots \leq u^{\star} \,\,\,\textrm{in} \,\,\, \Omega_T.
\end{equation}
Indeed, since $u_0$ is a viscosity sub-solution of \eqref{3.1} and $u_1$ solves \eqref{3.2} with $k=0$, we have
$$
G_{\psi_{0}}[u_{1}] -\partial_t u_1= 0 \leq  G_{\psi_{0}}[u_{0}] -\partial_t u_0
$$
in the viscosity sense. Moreover, $u_1=u_0=\varphi$ on $\partial_p \Omega_T$, so the comparison principle (see \cite{LW}) gives $u_0\leq u_{1}$ in $\Omega_T$. Assume inductively that we have verified that $u_{k-1}\leq u_k$ in $\Omega_T$. Since $h$ is increasing, having in mind the inductive assumption and the fact that $u_{k+1}$ is a solution of \eqref{3.2}, we conclude
$$
G_{\psi_{k}}[u_{k+1}] - \partial_t u_{k+1}= 0 \leq  G_{\psi_{k}}[u_{k}]- \partial_t u_k
$$
in the viscosity sense. Also $u_{k+1}=u_k=\varphi$ on $\partial_p \Omega_T$. Applying once more the comparison principle, we get $u_k\leq u_{k+1}$. Analogously, one can also show that $u_k\leq u^\star$, $\forall k\geq0$.

Using \eqref{3.3}, we define the pointwise limit 
$$u:=\lim\limits_{k\to\infty} u_{k}.$$
For any $Q\Subset \Omega_T$,  there exists a constant $C$ (depending only on $\mu$, $\|u_{\star}\|_{L^{\infty}(Q)}$, $\|u^{\star}\|_{L^{\infty}(Q)}$ and $\|f\|_{L^{\infty}(Q)}$) such that 
$$ |F(x,t,D^2u_{k}) -\partial_t u_{k}| \leq C \quad \mathrm{in} \ Q$$
in the viscosity sense, $\forall k\geq 0$. Therefore, $u_k$ is locally uniformly H\"{o}lder continuous (see \cite{LW}). By the Arzel\`a--Ascoli Theorem, it converges, up to a subsequence, locally uniformly in $\Omega_T$. Invoking stability arguments (see \cite{LW,Wang2}) and passing to the limit as $k\to\infty$, we conclude that $u$ is a viscosity solution of
$$
F(x,t, D^2u)-\partial_tu=g(u)+f.
$$

To conclude the proof, it remains to check that $u=\displaystyle\inf_{v\in\mathcal{S}}v$. Obviously, $u\in\mathcal{S}$. Let $v\in\mathcal{S}$; since
$$
G_{\psi_{k}}[u_{k+1}]- \partial_t u_{k+1} = 0\geq G_{\psi_{k}}[v]-\partial_t v
$$
in the viscosity sense, arguing as above, we get $v\geq u_{k+1}$, $\forall k\geq0$. Passing to the limit as $k\to\infty$ we conclude that $u=\displaystyle\inf_{v\in\mathcal{S}}v$.
\end{proof}

As a consequence of this result, we get the existence of solutions of \eqref{Equation Pe}. The Alexandrov-Bakelman-Pucci (ABP) estimate then implies their uniform boundedness.

\begin{theorem}\label{t3.2}
If $(A1)$-$(A4)$ hold, then the problem \eqref{Equation Pe} has a solution and
\begin{equation}\label{3.4}
    0 \le u_{\varepsilon} \le  \Upsilon \quad \mbox{in} \quad \Omega_T,
\end{equation}
where $\Upsilon=\Upsilon(\lambda, \Lambda, n, \|\varphi\|_\infty, c_0)$.
\end{theorem}

\begin{proof}
In order to apply Theorem \ref{t3.1}, we choose $u_{\star}$ as a solution of
$$
     \left \{
        \begin{array}{rclcl}
           F(x,t,D^2u_\star) -\partial_tu_\star&=& \xi & \textrm{ in }& \Omega_T \\
            u_\star&=& \varphi\ & \textrm{ on } & \partial_p \Omega_T,
        \end{array}
    \right.
$$
and $u^\star$ as a solution of
$$
     \left \{
        \begin{array}{rclcl}
           F(x,t,D^2u^\star) -\partial_tu^\star&=& 0 & \textrm{ in }& \Omega_T \\
            u^\star&=& \varphi\ & \textrm{ on } & \partial_p \Omega_T,
        \end{array}
    \right.
$$
where $\xi:=\sup(\beta_\varepsilon+f_\varepsilon)$. Existence of these solutions is a consequence of standard Perron's method. By construction, $u_{\star}$ and $u^{\star}$ are viscosity sub- and super-solutions of \eqref{Equation Pe}, respectively. A direct application of Theorem \ref{t3.1}, with $g=\beta_{\varepsilon}$, $f=f_\varepsilon$, gives the existence of a solution of \eqref{Equation Pe}.

To prove \eqref{3.4}, let $v_{\varepsilon}:= u_{\varepsilon} -\|\varphi\|_\infty$. Note that $v_{\varepsilon} \le 0$ on $\partial_p\Omega_T$ and from \eqref{Pucci} one has
\begin{eqnarray*}
    \mathcal{P}^{+}_{\frac{\lambda}{n}, \Lambda}(D^2 v_{\varepsilon}) - \partial_tv_{\varepsilon}& \ge& F(x,t, D^2 v_\varepsilon) -\partial_tv_{\varepsilon}=   F(x,t,D^2u_\varepsilon)-\partial_tu_{\varepsilon}\ge c_0.
\end{eqnarray*}
This means that $v_{\varepsilon} \in\underline{\mathscr{S}}(\frac{\lambda}{n}, \Lambda,c_0)$. The ABP estimate (\cite[Theorem 3.14]{Wang}) then implies 
$$\displaystyle\sup_{\Omega_T} (v_{\varepsilon})^{+} \le C(\lambda, \Lambda, n, c_0).$$
Thus, $u_{\varepsilon} \le \|\varphi\|_{\infty} + C(\lambda, \Lambda, n, c_0) =: \Upsilon$.

In order to prove the nonnegativity of $u_{\varepsilon}$ we assume the contrary, i.e. that $A_\varepsilon := \{(x,t) \in\Omega_T  : u_{\varepsilon}(x,t) < 0\}\neq\emptyset$. Since $\beta_\varepsilon$ is supported in
$[0,\varepsilon]$, then
$$
    \mathcal{P}^{-}_{\frac{\lambda}{n},\Lambda}(D^2 u_{\varepsilon}) - \partial_tu_{\varepsilon} \le F(x,t,D^2u_\varepsilon)-\partial_tu_{\varepsilon} = f_{\varepsilon}\le c_1\,  \textrm{ in }\, A_\varepsilon,
$$
which means that $u_\varepsilon\in\mathscr{\overline{S}}(\frac{\lambda}{n},\Lambda,c_1)$. Another application of the ABP estimate provides that $u_\varepsilon \ge 0$ in $A_\varepsilon$, which is a contradiction.
\end{proof}

\section{Uniform Lipschitz regularity in space-time}\label{s4}

\hspace{0.3cm}In this section we show that the family $\{u_{\varepsilon}\}_{\varepsilon >0}$ of solutions of \eqref{Equation Pe} is locally uniformly bounded in the $\textrm{Lip}_{\loc}(1,1/2)$-norm. As a consequence, we show that the limit function $u$ is a solution of the free boundary problem \eqref{1.2}. The main result of this section is the following theorem.

\begin{theorem} \label{t4.1}
Let $\{u_{\varepsilon}\}_{\varepsilon >0}$ be a family of solutions of $\eqref{Equation Pe}$. Let $K \subset\Omega_T$ be compact and $\tau >0$ be such that $\mathcal{N}_{2 \tau}(K) \subset\Omega_T$. If $(A1)$-$(A4)$ hold, then there exists a constant $L=L(\tau, \|\varphi\|_\infty)$ such that
$$
	\|u_{\varepsilon}\|_{\textrm{Lip}(1,1/2)(K)}\le L.
$$
\end{theorem}
Theorem \ref{t4.1} will be an immediate consequence of the following two results. First, using a Bernstein type argument, we obtain the uniform boundedness of the gradients of solutions (Proposition \ref{p4.1}). Next, we show that uniform spatial Lipschitz continuity implies uniform H\"{o}lder continuity in time with exponent $1/2$ (Proposition \ref{p4.2}).
\subsection{Uniform spatial regularity.}
We start with the uniform Lipschitz regularity in the spatial variables.
\begin{proposition}\label{p4.1}
If  $\{u_{\varepsilon}\}_{\varepsilon >0}$ is a family of solutions of \eqref{Equation Pe}, and $(A1)$-$(A4)$ hold, then there exists a constant $L>0$, independent of $\varepsilon\in(0,1)$, such that
$$
| \nabla u_{\varepsilon}(x,t)| \le L,\,\, \forall(x,t) \in \Omega_T.
$$
\end{proposition}
\begin{proof}
Note that the regularity assumptions on $F$, $f_\varepsilon$ and $\varphi$ guarantee that solutions are locally of class $C^3$ (\cite[Theorem 2]{Wang3}). 

Now, since $\beta_\varepsilon=0$ in $\{u_\varepsilon\geq\varepsilon\}$, we conclude from up to the boundary parabolic regularity theory (see \cite[Theorem 4.19]{Wang} and \cite[Theorem 2.5]{Wang2})  that
$$
|\nabla u_\varepsilon|\leq C(\|u_\varepsilon\|_\infty+\|f_\varepsilon\|_{n+1}+\|\varphi\|_\infty),
$$
in this region, where $C$ does not depend on $\varepsilon$. The result then follows from $(A3)$ and \eqref{3.4} with $L=L(\Upsilon, c_1,C)$.

To prove the uniform Lipschitz regularity in $\{u_\varepsilon\leq\varepsilon\}$, it is enough to show that at the maximum point of
$$
v_\varepsilon:= \frac{1}{2} |\nabla u_\varepsilon|^2 + \frac{\Gamma}{2\varepsilon^2}u_\varepsilon^2,
$$
where $\Gamma>0$ is a constant (independent of $\varepsilon$) to be chosen later,  $|\nabla u_\varepsilon|$ can be controlled by a universal constant $C$, since then one can write
$$
|\nabla u_\varepsilon|^2\leq 2v_\varepsilon\leq C^2+\Gamma =:L^2.
$$
Let $(x_0,t_0)$ be a maximum point of $v_\varepsilon$ in $\{u_\varepsilon\leq\varepsilon\}$. From the uniform gradient estimate in $\{u_\varepsilon\geq\varepsilon\}$, we may assume that it is an interior point. We drop the subscript $\varepsilon$ in $v_\varepsilon$, $u_{\varepsilon}$ and $f_{\varepsilon}$ for convenience. Direct computation shows that
\begin{eqnarray*}
	D_i v &=&  \sum_{k} D_k u D_{ik }u + \Gamma\varepsilon^{-2} u D_i u, \\ \nonumber
	D_{ij} v &=& \sum_{k} (D_{kj}u D_{ki} u +D_k u D_{ijk}u)+\Gamma\varepsilon^{-2} (D_i u D_j u + u D_{ij} u),\\ \nonumber
	\partial_tv &=& \sum_{k} D_k u D_{k}\partial_t u + \Gamma\varepsilon^{-2} u\partial_tu, \nonumber
\end{eqnarray*}
where $D_k u=\partial u/\partial x_k$. Differentiating \eqref{Equation Pe} in the $k$-th direction one gets
\begin{equation}\label{4.1}
\sum_{i,j} F_{ij}(x,t,D^2 u)D_{ijk} u - D_k\partial_tu= \varepsilon^{-2} \beta' D_ku + D_kf,
\end{equation}
where $F_{ij}(\cdot,M):=\partial F/\partial m_{ij}$, $M=(m_{ij})$. The uniform ellipticity of $F$ implies that $A_{ij}:= F_{ij}(x_0,t_0, D^2 u(x_0,t_0))$ is a positive matrix, therefore at $(x_0,t_0)$ we have
\begin{eqnarray*}
 0 &\ge& \sum_{i,j} A_{ij} D_{ij} v-\partial_tv= \textrm{tr}(D^2u (A_{ij} D^2 u))  \\
 &+&  \sum_{k} D_k u \left(\sum_{i,j} A_{ij} D_{ijk} u\right)+\Gamma\varepsilon^{-2} \sum_{i,j} A_{ij} D_i u D_j u \\
  &+& \Gamma\varepsilon^{-2} u \sum_{i,j} A_{ij} D_{ij} u -\sum_{k}  D_k u D_{k}\partial_tu -\Gamma\varepsilon^{-2} u\partial_tu,
\end{eqnarray*}
which, together with \eqref{4.1}, provides
\begin{eqnarray*}
0 &\ge&  \textrm{tr}(D^2u (A_{ij} D^2 u)) + \sum_{k} D_k u \left(\sum_{i,j} A_{ij} D_{ijk}u\right)\\
&+& \Gamma \varepsilon^{-2} \sum_{i,j} A_{ij} D_i u D_j u +  \Gamma\varepsilon^{-2} u \sum_{i,j} A_{ij} D_{ij} u\\
&-& \sum_{k}  D_k u D_{k}\partial_tu -\Gamma\varepsilon^{-2} u\partial_tu\\
&\ge& \sum_{k} D_k u \left(D_k\partial_tu + \varepsilon^{-2} \beta' D_k u + D_kf\right) +\Gamma\varepsilon^{-2} \lambda |\nabla u|^2\\
&-&  \sum_{k} D_k u D_{k}\partial_tu +\Gamma\varepsilon^{-2} u \left(\sum_{i,j} A_{ij} D_{ij} u - \partial_tu\right)\\
&\ge& \varepsilon^{-2} \beta'|\nabla u|^2 + \sum_{k} D_ku D_kf +\Gamma\varepsilon^{-2} \lambda |\nabla u|^2 - \lambda |u| \varepsilon^{-2} \varepsilon^{-1} \beta\\
&=& \varepsilon^{-2}  \left(\beta'|\nabla u|^2 - \varepsilon^2 |\nabla u| |\nabla f|  +\Gamma  \lambda |\nabla u|^2 - \Gamma \varepsilon^{-1}\lambda |u| \beta\right).
\end{eqnarray*}
Therefore,
\begin{eqnarray}\label{4.2}
(\beta'(u/\varepsilon)+ \Gamma \lambda) |\nabla u|^2 -\varepsilon^2 |\nabla u| |\nabla f|  &\le& \Gamma\lambda |u| \varepsilon^{-1} \beta(u/\varepsilon).
\label{estimate}
\end{eqnarray}
By choosing $\Gamma:= \frac{2}{\lambda}\max|\beta'|$, from \eqref{4.2} we get
$$
C_1 |\nabla u|^2 - C_2 \varepsilon^2 |\nabla u| \le C_1 |u| \varepsilon^{-1}\beta(u/\varepsilon) \le C_1C_3,
$$
with $C_1=\max|\beta'|$, $C_2=\|\nabla f\|_\infty$ and $C_3 = \max |\beta|$, which leads to
$$
|\nabla u(x_0,t_0)| \le C,
$$
where $C$ depends only on dimension, ellipticity, $\|\beta\|_{C^1}$ and $\|\nabla f\|_{\infty}$, thus being independent of $\varepsilon$.
\end{proof}

As an immediate consequence we have the following result.

\begin{corollary}\label{c4.1}
Let $\{u_{\varepsilon}\}_{\varepsilon >0}$ be a family of  solutions of \eqref{Equation Pe}. Let $ K \subset\Omega_T$ be a compact set and $\tau >0$ be such that $\mathcal{N}^{-}_{\tau}(K) \subset\Omega_T$. If $(A1)$-$(A4)$ hold, then there exists a constant $L=L(\tau)$ such that
$$
|\nabla u_{\varepsilon}(x,t)| \le L, \quad \forall(x,t) \in K.
$$
\end{corollary}
\begin{proof}
For $(x_0,t_0) \in K$, consider the function
$$
    w_{\varepsilon,r}(x,t) := \frac{1}{r} u_{\varepsilon}(x_0 + rx, t_0+r^2 t).
$$
For $r\in(0,\tau)$ we have that $w_{\varepsilon,r}$ is a solution of
$$
F_{r} (x,t,D^2 w_{\varepsilon,r})- \partial_tw_{\varepsilon,r} = \beta_{\varepsilon / r} (w_{\varepsilon,r}) + rf_{\varepsilon}=:g_\varepsilon(x,t)
$$
in $B_{1} \times (-1,0)$, where $F_{r}(x,t,M) \colon = r F \left(x_0 + rx, t_0 + r^2 t, \frac{1}{r} M\right)$. The result now follows from Proposition \ref{p4.1}.
\end{proof}

\subsection{Uniform regularity in time.}
Next, as was mentioned above, using the uniform Lipschitz continuity in the space variables, we obtain the uniform H\"{o}lder continuity in time. First, we need the following lemma.
\begin{lemma}\label{l4.1}
Let $u \in C(\overline{B}_{1}(0) \times [0, 1/(4n+M_0)])$ be such that
$$
|F(x,t,D^2 u)-\partial_tu| \leq M_{0}\quad \textrm{in} \quad \{u>1\},
$$
for some $M_0 >0$, and $|\nabla u| \leq L$, for some $L>0$. Then there exists a constant $C=C(L)$ such that
$$
|u(0,t)-u(0,0)| \leq C, \quad \textrm{if} \quad 0 \leq t \leq \frac{1}{4n + M_0}.
$$
\end{lemma}
\begin{proof}
Without loss of generality we may assume that $L>1$. We divide the proof into two steps.\\
{\bf Step 1.} \, First we claim that, if
$$
Q_{t_{0},t_{1}} := B_{1}(0) \times (t_0,t_1) \subset  \{u> 1\} \quad \textrm{for} \quad  t_1-t_0 \leq \frac{1}{4n+M_0},
$$
then
$$
    |u(0,t_1)-u(0,t_0)| \leq 2L.
$$
In fact, let
$$
    h^{\pm}(x,t):= u(0,t_0) \pm L \pm \frac{2L}{\Lambda} |x|^2 \pm (4nL + M_0)(t-t_0).
$$
By \eqref{Pucci} one has
\begin{eqnarray*}
    \partial_th^{+} - F(x,t,D^2 h^{+}) &\geq& \partial_th^{+} - \mathcal{P}^{+}_{\frac{\lambda}{n}, \Lambda}(D^2 h^{+})\\
    &=& \partial_th^{+} - \left( \Lambda \sum_{e_{i}>0}e_{i} + \frac{\lambda}{n} \sum_{e_{i} <0} e_{i}\right)\\
    &=& (4nL + M_0) - \Lambda \frac{4Ln}{\Lambda} = M_0,
\end{eqnarray*}
and
\begin{eqnarray*}
    \partial_th^{-} - F(x,t,D^2 h^{-}) &\leq& \partial_th^{-} - \mathcal{P}^{-}_{\frac{\lambda}{n}, \Lambda}(D^2 h^{-})\\
    &=& \partial_th^{-} - \left( \frac{\lambda}{n} \sum_{e_{i}>0}e_{i} + \Lambda \sum_{e_{i} <0} e_{i}\right)\\
    &=& -(4nL + M_0) - \Lambda \frac{-4Ln}{\Lambda} =- M_0.
\end{eqnarray*}
Set
$$
	t_2 := \sup_{t_0 \le \bar{t} \le t_1} \{\bar{t} : |u(0,t)-u(0,t_0)| \le 2L, \,\, \forall \, t_0 \le t \le \bar{t}\}.
$$
So  $t_{0} < t_{2} \leq t_{1}$ is such that
$$
    |u(0,t)-u(0,t_0)| \leq 2L, \quad \textrm{for} \quad t \in [t_0,t_2).
$$
Moreover, from the Lipschitz continuity in space, one has
$$
    h^{-} \leq u \leq h^{+} \quad \textrm{on}\quad\partial_{p}Q_{t_0,t_2}.
$$
On the other hand,
\begin{eqnarray*}
    \partial_th^{-} - F(x,t,D^2 h^{-}) &\leq& -M_0 \leq \partial_tu - F(x,t,D^2 u)\\
     &\leq& M_{0} \leq \partial_th^{+} - F(x,t,D^2 h^{+}).
\end{eqnarray*}
Therefore,
$$
    h^{-} \leq u \leq h^{+} \quad \textrm{in}\quad Q_{t_0,t_2}.
$$
In particular, since $t_2-t_0 \leq t_1-t_0 \leq \frac{1}{4n+M_0}$ and $L>1$ one has
$$
    |u(0,t_2)-u(0,t_0)| < 2L.
$$
Because of the strict inequality above, we may take $t_2=t_1$ and therefore the claim is proved.\\
{\bf Step 2.}\, Let us consider now the cylinder $Q_{0,t}$ with $0 < t \leq \frac{1}{4n + M_0}$.\\
If $Q_{0,t} \subset \{u>1\}$, we apply Step 1 to get
$$
    |u(0,t)-u(0,0)| \leq 2L.
$$
If $Q_{0,t} \nsubseteq  \{u>1\}$, let $0 \leq t_{1} \leq t_{2} \leq t$ and $x_{1},x_{2} \in \overline{B}_{1}(0)$ be such that
$$
    0 \leq u(x_1,t_1) \leq 1, \,\,\,\, 0 \leq u(x_2,t_2) \leq 1
$$
and
$$
    (\overline{B}_{1}(0) \times (0,t_{1})) \cup (\overline{B}_{1}(0) \times (t_{2},t)) \subset  \{u>1\}.
$$
Then, Step 1 and the Lipschitz continuity in space provide
\begin{eqnarray*}
    |u(0,t)-u(0,0)| &\leq& |u(0,t)-u(0,t_2)| + |u(0,t_2) - u(x_2,t_2)| + |u(x_2,t_2)|\\
    &+& |u(x_1,t_1)| + |u(x_1,t_1)-u(0,t_1)| + |u(0,t_1)-u(0,0)|\\
    & \leq& 2(2L+L+1),
\end{eqnarray*}
which completes the proof.
\end{proof}
We are now ready to prove uniform H\"{o}lder continuity of solutions in time.

\begin{proposition}\label{p4.2}
Let $\{u_{\varepsilon}\}_{\varepsilon >0}$ be a family of solutions of $\eqref{Equation Pe}$. Let $K \subset\Omega_T$ be compact and $\tau >0$ be such that $\mathcal{N}_{2 \tau}(K) \subset \Omega_T$. If $(A1)$-$(A4)$ hold, then there exists a constant $C>0$, independent of $\varepsilon$, such that
$$
|u_{\varepsilon}(x,t+\Delta t) -u_{\varepsilon}(x,t)| \leq C |\Delta t|^{1/2},\,\,\,\textrm{ for }(x,t),(x,t+\Delta t) \in K.
$$
\end{proposition}
\begin{proof}
Let $r\in(0,\tau)$, $(x_0,t_0) \in K$ and $w_{\varepsilon,r}(x,t)$, $g_\varepsilon(x,t)$ be as in the proof of Corollary \ref{c4.1}.
From $(A2)$ and $(A3)$ we get, in the set $\{ w_{\varepsilon,r} >1\}$,
$$
0 \le g_{\varepsilon}(x,t) \le (1 + rc_1) \le (1+\tau c_1)=:C_{\star}.
$$
Also $|\nabla w_{\varepsilon,r}(x,t)|\leq L$. Therefore, we may apply Lemma \ref{l4.1}, with $M_0 = C_{\star}$, to obtain
$$
|w_{\varepsilon,r}(0,t) - w_{\varepsilon,r}(0,0)| \leq C, \quad \textrm{for}\quad 0 \leq t \leq \frac{1}{4n + C_\star},
$$
or in other terms
$$
|u_\varepsilon(x_0,t_0+r^2t)-u_\varepsilon(x_0,t_0)|\leq Cr,\,\,\,\textrm{ for }0\leq t\leq\frac{1}{4n+C_\star}.
$$
In particular, for $r\in(0,\tau)$, one has
\begin{equation}\label{4.3}
\left|u_{\varepsilon} \left(x_0, t_0 + \frac{r^2}{4n + C_\star}\right) - u_{\varepsilon} (x_0,t_0)\right| \leq Cr.
\end{equation}
Now if $(x_0,t_0+\Delta t) \in K$ and $0 < \Delta t < \frac{r^2}{4n + C_\star}$, taking $r = \Delta t^{1/2} \sqrt{4n+C_\star}$ in \eqref{4.3} leads to
$$
|u_{\varepsilon}(x_0,t_0 + \Delta t) - u_{\varepsilon}(x_0,t_0)| \leq C\sqrt{4n+C_\star} \Delta t^{1/2}.
$$
On the order hand, if $\Delta t \geq \frac{r^2}{4n+C_\star}$, from \eqref{3.4} we get
$$
|u_{\varepsilon}(x_0,t_0 + \Delta t) - u_{\varepsilon}(x_0,t_0)| \leq 2 \Upsilon \leq \frac{2 \Upsilon}{\tau} \sqrt{4n+C_\star} \Delta t^{1/2},
$$
which completes the proof.
\end{proof}

\section{The limiting free boundary problem}\label{s5}

\hspace{0.3 cm} We start this section by letting $\varepsilon \to 0$ in \eqref{Equation Pe}. Recalling Theorem \ref{t4.1},  we know that up to a subsequence, there exists a limiting function $u$, obtained as the uniform limit of $u_{\varepsilon}$ as $\varepsilon \to 0$. We now show that $u$ is a viscosity solution of \eqref{1.2}, where $f$ is the uniform limit of $f_\varepsilon$.
\begin{theorem}\label{t5.1}
Let $\{u_{\varepsilon}\}_{\varepsilon >0}$ be a family of solution of \eqref{Equation Pe}. If $(A1)$-$(A4)$ hold then, up to a subsequence,
\begin{enumerate}
\item[(1)] $u_{\varepsilon} \to u$ locally uniformly in $\Omega_T$ and $u\in\textrm{Lip}_{\loc}(1,1/2)(\Omega_T)$;
\item[(2)] $u$ is a solution of \eqref{1.2}, where $f$ is the uniform limit of $f_\varepsilon$;
\item[(3)] the function $t \mapsto u(x,t)$ is non-decreasing in time.
\end{enumerate}
\end{theorem}
\begin{proof}
Parts (1) and (2) follow from Theorem \ref{t4.1} and the Arzel\`a--Ascoli Theorem.  In fact, since $u_{\varepsilon}\in \textrm{Lip}_{\loc}(1,1/2)(\Omega_T)$, with a uniform estimate, we can pass to the limit (up to a subsequence) and obtain a function
$$
u(x,t) = \lim_{\varepsilon\to 0} u_{\varepsilon}(x,t),
$$
the convergence being uniform on compact subsets of $\overline{{\Omega}_{T}}$. Hence, 
$$u\in\textrm{Lip}_{\textrm{loc}}(1,1/2)(\Omega_T).$$
Moreover, $u$ is a viscosity solution of \eqref{1.2}. Indeed, if $u(x_0,t_0)=c>0$, then using the uniform convergence $u_{\varepsilon} \to u$ and the equicontinuity of $u_{\varepsilon}$, we conclude that for every small $\varepsilon$ one has, in a small neighbourhood of $(x_0,t_0)$, that $u_{\varepsilon} \ge \frac{c}{2} > \varepsilon$. So $\beta_{\varepsilon}(u_{\varepsilon}) =0$. Since $f_{\varepsilon} \to f$, invoking stability arguments (\cite{LW,Wang2}) and passing to the limit in \eqref{Equation Pe}, we conclude that $u$ is a solution of $\eqref{1.2}$.

In order to check (3), we define, for $t>0$ and $h>0$, 
$$u_h(\cdot,t)  :=u(\cdot,t+h); \quad \ f_{h}(\cdot,t)  := f(\cdot,t+h); \quad \ \varphi_h(\cdot,t) := \varphi(\cdot,t+h) $$
and $F_h(\cdot,t,\cdot):=F(\cdot,t+h,\cdot)$.
Set also $\varphi_h(x,0):=\varphi(x,0)=0$. Since $u$ is a solution of \eqref{1.2}, then $u_h$ is a solution of the same problem with $F=F_h$, $f=f_h$ and $\varphi=\varphi_h$. From $(A4)$ we know that $\varphi$ is non-decreasing in $t$ and $\varphi(x,0)=0$, therefore  $u_{h} \ge u$ on $\partial_p \Omega_T$. Observe that $(A3)$ provides $f_h(x,t)\le f(x,t)$. Since also $u \ge 0$, we can apply a comparison argument to verify that $u_{h} \ge u$ in $\Omega_T$, so the function $t\mapsto u(x,t)$ is non-decreasing.
\end{proof}

\section{Porosity of the free boundary}\label{s6}

In this section we establish the exact growth of the solution near the free boundary, from which we deduce the porosity of its time level sets.

\begin{definition}\label{d5.1}
A set $E\subset\mathbb{R}^n$ is called porous with porosity $\delta>0$, if there exists $R>0$ such that
$$
   \forall x\in E, \,\,\,\forall r\in(0,R),\,\,\,\exists y\in\mathbb{R}^n\,\textrm{ such that }\,B_{\delta r}(y)\subset B_r(x)\setminus E.
$$
\end{definition}
A porous set of porosity $\delta$ has Hausdorff dimension not exceeding $n-c\delta^n$, where $c=c(n)>0$ is a constant depending only on $n$. In particular, a porous set has Lebesgue measure zero.

The following theorem is the main result of this section.

\begin{theorem}\label{t6.1}
Let $u$ be a solution of \eqref{1.2}. If (A1) holds and $f$ satisfies (A3) then, for every compact set $K \subset \Omega_T$ and every $t_0 \in (0,T)$, the set 
$$\partial\{u > 0\}\cap K \cap\{t=t_0\}$$ 
is porous in $\mathbb{R}^n$, with porosity depending only on $\Upsilon$ and $\textrm{dist}(K,\partial_p \Omega_T)$.
\end{theorem}
To prove the theorem we need to prove some auxiliary results.
\subsection{Non-degeneracy}
We start by proving a non-degeneracy result. Let us remark that, without loss of generality, we may consider in what follows the domain $Q_1=Q_1(0,0)$ instead of $Q_1(z,s)$.

\begin{lemma}\label{l6.1}
Let $u \in C(Q_1)$ be a solution of
$$
F(x,t , D^{2} u) - \partial_t u = f \quad \textrm{in} \quad \{u>0\},
$$
with $f$ satisfying the lower bound in (A3). Then for every $(z,s) \in \overline{\{u>0\}}$ and $r >0$ with $Q_r(z,s) \subset Q_1$ we have
$$
\sup_{(x,t) \in \partial_p Q^{-}_{r}(z,s)} u(x,t) \ge \mu_0 r^{2} + u(z,s),
$$
where $\mu_0 = \min \left(\frac{c_0}{2}, \frac{c_0}{4n \Lambda}\right)$.
\end{lemma}
\begin{proof}
Suppose that $(z,s) \in\{u>0\}$, and, for small $\delta>0$, set
$$
\omega_{\delta}(x,t):= u(x,t) - (1-\delta)u(z,s)\,\,\textrm{ and }\,
\psi(x,t):= \frac{c_0}{4n \Lambda}|x-z|^{2} - \frac{c_0}{2}(t-s).
$$
Since $D_{ij} \psi = \frac{c_0}{2n \Lambda}\delta_{ij}$ then, from \eqref{Pucci}, one has
\begin{eqnarray*}
F(x,t, D^2 \psi) - \partial_t \psi &\le&   \mathcal{P}^{+}_{\frac{\lambda}{n}, \Lambda}(D^2 \psi) -\partial_t \psi  \\
&=& \Lambda \sum_{e_i >0} e_i + \frac{\lambda}{n} \sum_{e_i <0} e_i  + \frac{c_0}{2} \\
&=&  \Lambda \frac{n c_0}{2n \Lambda} + \frac{c_0}{2}=c_0\\
	&\le&  f(x,t)  =  F(x,t, D^2 u) - \partial_t u\\
	&=&  F(x,t, D^2 \omega_{\delta})-\partial_t \omega_{\delta}.
\end{eqnarray*}
Moreover, $\omega_{\delta} \le \psi$ on $\partial\{u>0\}\cap Q^{-}_{r}(z,s)$. Note that we can not have
$$
\omega_{\delta} \le \psi \quad \textrm{on} \quad \partial_{p}Q^{-}_{r}(z,s)\cap\{u>0\},
$$
because otherwise we could apply the comparison principle  to obtain
$$
\omega_{\delta} \le \psi \quad \textrm{in} \quad Q^{-}_{r}(z,s)\cap\{u>0\},
$$
which contradicts the fact that $\omega_{\delta}(z,s) = \delta u(z,s) >0 = \psi(z,s)$. Hence, for $(y,\tau) \in \partial_{p}Q^{-}_{r}(z,s)$ we must have
$$
\omega_{\delta}(y, \tau) > \psi(y,\tau) = \mu_0 r^{2}.
$$
Letting $\delta \to 0$ in the last inequality we conclude the proof.
\end{proof}

\subsection{A class of functions in the unit cylinder}

Next, we establish the growth rate of the solution near the free boundary, which is known for $p$-parabolic variational problems (see \cite{Sha}) but is new in the fully nonlinear framework. We start by introducing a class of functions.
\begin{definition}
We say that a function $u \in C(Q_1)$ is in the class $\Theta$ if $0\le u\le1$ in $Q_1$, 
$\| F(x,t,D^2 u)-\partial_t u\|_{\infty} \le 1$ in $Q_1$,
in the viscosity sense and, moreover, $ \partial_t u\ge0$ and $u(0,0)=0$.
\end{definition}
\noindent Note that the last two conditions make sense due to the regularity of $u$ guaranteed by the first two (\cite{Wang, Wang2}).

In order to proceed, we need to introduce some notation. Set
$$
S(r,u,z,s):=\displaystyle \sup_{Q^{-}_{r}(z,s)}u.
$$
For $u\in \Theta$, we define
$$
H(u,z,s):= \left\{ j \in \mathbb{N} \cup \{0\} : S(2^{-j},u,z,s) \le M  S(2^{-j-1},u,z,s)\right\},
$$
where $M:=4\max(1,\frac{1}{\mu_0})$, with $\mu_0$ as in Lemma \ref{l6.1}. When $(z,s)$ is the origin, we suppress the point dependence.

The following lemma is the main step towards the growth control of the solution near the free boundary.
\begin{lemma}\label{l6.2}
If $u\in \Theta$, then there is a constant $C_1 = C_1(n,c_1)>0$ such that
$$
S(2^{-j-1},u) \le C_12^{-2j},\,\,\,\,\forall j\in H(u).
$$
\end{lemma}
\begin{proof}
First, note that $H(u)\not=\emptyset$ because $0\in H(u)$. Indeed, using Lemma \ref{l6.1}, we have
$$
S(1,u)\leq1=4\left(\frac{1}{\mu_0}\right)\mu_02^{-2}\leq4\left(\frac{1}{\mu_0}\right)S(2^{-1},u)\leq MS(2^{-1},u).
$$
Next, suppose the conclusion of the lemma fails. Then, for every $k \in \mathbb{N}$, there is $u_k \in \Theta$ and $j_k \in H(u_k)$ such that
$$
S(2^{-j_k-1},u_k) \ge k 2^{-2j_k}.
$$
Define $v_k:Q_1 \rightarrow \mathbb{R}$ by
$$
v_k(x,t):= \frac{u(2^{-j_k}x, 2^{-2j_k}t)}{S(2^{-j_k-1},u_k)}.
$$
One easily verifies that
$$0 \le v_k \le 1 \quad \textrm{in} \ Q^{-}_1; \qquad \| F_k(x,t, D^2 v_k)-\partial_t v_k\|_{\infty} \le \frac{c_1}{k}; $$
$$\sup_{Q^{-}_{1/2}}v_k =1;\qquad v_k(0,0)=0; \qquad\partial_t v_k \ge 0 \quad \textrm{in} \ Q_1^-,$$
where
$$
F_k(x,t, M):= \frac{2^{-2j_k}}{S(2^{-j_k-1},u_k)} F \left(2^{-j_k}x,(2^{-j_k})^2 t,\frac{S(2^{-j_k-1},u_k)}{2^{-2j_k}} M\right)
$$
is a uniform $(\lambda, \Lambda)$-elliptic operator. Using compactness arguments (see \cite{Wang,Wang2}), we infer that there is a subsequence of $v_k$ converging locally uniformly in $Q^{-}_{1}$ to a function $v$. Moreover,
$$
\mathcal{F}(x,t,D^2 v) - \partial_t v =0, \quad v(0,0)=0, \quad v \ge 0, \quad \partial_t v \ge 0
$$
in $Q^{-}_{1}$ for some $(\lambda,\Lambda)$-elliptic operator $\mathcal{F}$. The strong maximum principle (see \cite{LI04}) then implies that $v\equiv0$, which contradicts  the fact that
$$\sup_{Q^{-}_{1/2}}v =1.$$
\end{proof}
We are now ready to prove the growth control of the solution near the free boundary.
\begin{lemma} \label{l6.3}
If $u\in \Theta$, then there is a constant $C_0 = C_0(n,L,c_1)>0$ such that
$$
	|u(x,t)| \le C_0 (d(x,t))^2, \quad \forall \,\, (x,t) \in Q_{1/2},
$$
where
$$
d(x,t):=
\begin{cases}
  \sup \left\{ r : Q_r(x,t)\subset\{u>0\}\right\}, & \mbox{if } (x,t)\in\{u>0\} \\
  0, & \mbox{otherwise}.
\end{cases}
$$
\end{lemma}
\begin{proof}
It suffices to show that 
\begin{equation}\label{6.1}
S(2^{-j},u)\leq 4C_12^{-2j},\,\,\,\forall j\in\mathbb{N}.
\end{equation}
In fact, for a fixed $r\in(0,1)$, by choosing $j\in\mathbb{N}$ such that $2^{-j-1}\leq r\leq 2^{-j}$, one has
\begin{equation}\label{6.2}
\sup_{Q_r^-(0,0)}u\leq\sup_{Q_{2^{-j}}^-}u\leq 4C_12^{-2j}
=16C_12^{-2j-2}\leq16C_1r^2.
\end{equation}
In order to prove \eqref{6.1}, let us take the first $j$ for which it fails (if there is no such $j$, we are done). Then
$$
S(2^{-(j-1)},u)\le 4C_1 2^{-2(j-1)} < 4S(2^{-j},u) \le M S(2^{-j},u),
$$
so $j-1 \in H(u)$, and we can apply Lemma \ref{l6.2} to reach the contradiction
$$
S(2^{-j},u)\le C_1 2^{-2(j-1)} = 4C_12^{-2j}.
$$

To obtain a similar estimate for $u$ over the whole cylinder (and not only over its lower half) we use a barrier from above. Set
$$
\omega(x,t):= A_1|x|^2 + A_2t,
$$
where $A_2 = 2\Lambda n A_1$ and $A_1 >0$. Then in $Q^{+}_{1} = B_{1}(0) \times (0,1)$ one gets from \eqref{Pucci} that
\begin{eqnarray*}
	F(x,t, D^2 \omega) - \partial_t \omega &\le& \mathcal{P}^{+}_{\frac{\lambda}{n}, \Lambda}(D^2 \omega) - \partial_t \omega \\
	&=& \Lambda \sum_{e_i >0} e_i + \frac{\lambda}{n} \sum_{e_i<0} e_i - A_2\\
	&=& 2n \Lambda A_1 - A_2 = 0\le F(x,t, D^2 u) - \partial_t u.
\end{eqnarray*}
If $A_1$ is large enough, then $\omega \ge u$ on $\partial_p Q^{+}_{1}$, where for the estimate on $\{t=0\}$ we used $S(r,u)\le 16C_1r^2$ from \eqref{6.2}. Hence, by the comparison principle one has $\omega \ge u$ in $Q^{+}_{1}$. Therefore
$$
\sup_{Q_r(0,0)} u \le C_0 r^{2},
$$
for a constant $C_0>0$.
\end{proof}

\subsection{Porosity of the free boundary in time levels}

We close the paper by proving Theorem \ref{t6.1}.
\begin{proof}
Without loss of generality, we assume that $K$ is the closed unit cylinder $\overline{Q}_1$, and $\overline{Q}_{2} \subset \Omega_T$. For $(x,t) \in \{u>0\} \cap \overline{Q}_{1}$, let $d(x,t)$ be as in Lemma \ref{l6.3} and take $(x_0,t_0) \in \partial\{u>0\} \cap \overline{Q}_{1}$ to be the point where the distance is attained. Define
$$
v(y,s):= u(x_0+y,t_0+s), \quad \textrm{for} \,\,\, (y,s) \in Q_1.
$$
We have
$$
\|F(x,t, D^2 v) - \partial_t v\|_{\infty} \le c_1, \quad 0 \le v \le\Upsilon, \quad v(0,0)=0,
$$
hence, if $\kappa=\max\{1,c_1,\Upsilon\}$, then
$(1/\kappa)v(y,s) \in \Theta$. Lemma \ref{l6.3} then provides \begin{equation}\label{6.3}
u(x,t)=v(x-x_0, t-t_0) \le  \kappa C_0 (d(x,t))^2.
\end{equation}
Now if $(z,\tau)\in\partial\{u>0\} \cap \overline{Q}_{1}$, then for $r\in(0,1)$, using Lemma \ref{l6.1}, and the fact that $\partial_tu\geq0$ in $Q_1$, one concludes that there exists $x_1\in\partial B_r(z)$, such that
$$
u(x_1,\tau)\ge\mu_0 r^2.
$$
Together with \eqref{6.3}, we have
$$
\mu_0 r^2 \le u(x_1,\tau) \le \kappa C_0 (d(x_1,\tau))^2,
$$
which implies that
$$
d(x_1,\tau) \ge \delta r, \quad \delta = \sqrt{\frac{\mu_0}{\kappa C_0}}
$$
and hence
$$
B_{\delta r}(x_1)\subset B_{d(x_1,\tau)}(x_1)\subset\{u>0\}.
$$
Note that $\delta \le 1$. We claim now that there is a ball
\begin{equation}\label{6.4}
B_{\frac{\delta}{2}r}(y) \subset B_{\delta r}(x_1) \cap B_r(z) \subset B_r(z) \setminus \partial\{u>0\},
\end{equation}
which means that the set $\partial\{u>0\}\cap\{t=\tau\} \cap \overline{B}_{1}$ is porous with porosity constant $\delta/2$.
 
To check \eqref{6.4} we choose $y\in[z,x_1]$ such that $|y-x_1|=\delta r/2$. For each $\xi\in B_{\frac{\delta}{2}r}(y)$ one has
$$
|\xi-x_1|\leq|\xi-y|+|y-x_1|\leq\frac{\delta}{2}r+\frac{\delta}{2}r=\delta r
$$
and, since $x_1\in\partial B_r(z)$, also
$$
|\xi-z|\leq|\xi-y|+|z-x_1|-|y-x_1|\leq\frac{\delta}{2}r+r-\frac{\delta}{2}r=r,
$$
and therefore \eqref{6.4} is true.
\end{proof}

\section{The free boundary condition} \label{s7}

In this section we study the behavior of the limiting function $u$ near the free boundary $\partial\{u>0\}$. We will assume that $F$ is rotational invariant with respect to the Hessian, i.e.,
\begin{equation}\label{7.1}
F(x,t,M)=F(x,t,OMO^{-1}),\,\,\forall O\in O(n),
\end{equation}
where $O(n)$ is the set of orthogonal matrices. In other words, $F$ depends on the eigenvalues of the Hessian. Additionally, we assume that the singular reaction term $\beta_\varepsilon$ is built as an approximation of unity, i.e.,
\begin{equation}\label{7.2}
\beta_\varepsilon(s):=\varepsilon^{-1}\beta(\varepsilon^{-1}s),
\end{equation}
where $\beta$ is a nonnegative smooth real function supported in $[0,1]$ such that
$$
\|\beta\|_\infty\leq1\,\,\textrm{ and }\,\,L:=\int_0^1\beta(s)\,ds<\infty.
$$
As noted in \cite{LW06}, when $f_\varepsilon\not\equiv0$ in \eqref{Equation Pe}, there exist limits $u=\displaystyle \lim_{\varepsilon\rightarrow0} u_\varepsilon$ which degenerate even when they have a smooth free boundary $\partial\{u>0\}$. When $f_\varepsilon\equiv0$ this can not happen  because of Hopf's principle. This shows that in general one cannot expect the free boundary condition to hold for any limit $u$. To keep things simple, we will assume from now on that $f_\varepsilon\equiv0$ in \eqref{Equation Pe}.
\begin{remark}\label{r7.1}
Note that Theorems \ref{t4.1} and \ref{t5.1} remain true when $f_\varepsilon\equiv 0$ and $\beta_\varepsilon$ is built as an approximation of unity.\
\end{remark}
Next, we consider the elliptic operator $F^*$, defined as the pointwise limit of $\varepsilon F(x,t,\varepsilon^{-1}M)$, i.e.,
\begin{equation}\label{7.3}
F^*(x,t,M):=\lim_{\varepsilon\rightarrow0}\varepsilon F(x,t,\varepsilon^{-1}M).
\end{equation} 
Observe, that in general such limit may not exist. However, when $F$ is concave or convex, then the limit exists. In fact, from the dominated convergence theorem one gets a stronger assertion (see \cite[Proposition 6.1]{RT11} for details): if $F$ satisfies $(A1)$, then
\begin{equation}\label{7.4}
\exists\,\,\,\,\lim_{\|M\|\rightarrow\infty}F_{ij}(x,t,M):=F^*_{ij}(x,t),
\end{equation}
and
$$
F^*(x,t,M)=tr(F^*_{ij}(x,t)M).
$$
Note that if $u_\varepsilon$ is a solution of 
$$
F(x,t,D^2u_\varepsilon)-\partial_tu_\varepsilon=\beta_\varepsilon(u_\varepsilon),
$$
then
\begin{equation}\label{7.5}
a^\varepsilon_{ij}(x,t)D_{ij}u_\varepsilon-\partial_tu_\varepsilon=\beta_\varepsilon(u_\varepsilon),
\end{equation}
where
\begin{equation}\label{7.6}
a^\varepsilon_{ij}(x,t):=\displaystyle\int_0^1F_{ij}(x,t,sD^2u_\varepsilon)\,ds.
\end{equation}
Moreover, from \cite{Wang}-\cite{Wang3}, we have enough regularity to argue as in the proof of \cite[Lemma 6.3]{RT11} to obtain the following result.
\begin{proposition}\label{p7.1}
If $u_\varepsilon$ is a solution of \eqref{Equation Pe} with $f_\varepsilon=0$ such that $u_\varepsilon\rightarrow u$ uniformly on compact subsets of $\Omega_T$ as $\varepsilon\rightarrow0$, then $a^\varepsilon_{ij}(x,t)$ converges pointwise to a uniformly elliptic matrix
$$
b_{ij}(x,t):=
\begin{cases}
  \displaystyle\int_0^1F_{ij}(x,t,sD^2u)\,ds, & \mbox{if } (x,t)\in\{u>0\} \\
  F^*_{ij}(x,t), & \mbox{if }(x,t)\in\partial\{u>0\},
\end{cases}
$$
where $F^*_{ij}$ and $a_{ij}^\varepsilon$ are defined by \eqref{7.4} and \eqref{7.6}, respectively.

\end{proposition}
In order to state the next result, we need to define blow-ups. The function
$$
v_\lambda(x,t):=\frac{1}{\lambda}v(x_0+\lambda x,t_0+\lambda^2t)
$$
will be referred to as a blow up of $v$ at $(x_0,t_0)$.
The following lemma deals with convergence for blow-up limits. It is a generalization of \cite[Lemma 4.4]{LW06}. Its proof is classical and we will omit it, referring the reader to \cite{Caf2,LW06} for details.
\begin{lemma}\label{l7.1}
Let $u_\varepsilon$ be a solution of \eqref{Equation Pe} with $f_\varepsilon=0$ such that $u_\varepsilon\rightarrow u$ uniformly on compact subsets of $\Omega_T$ as $\varepsilon\rightarrow0$. Let $(u_\varepsilon)_{\lambda_k}$ and $u_{\lambda_k}$ be blow-ups of $u_\varepsilon$ and $u$ at $(x_k,t_k)$, respectively, for $\lambda_k\rightarrow0$. If $(x_k,t_k)$, $(x_0,t_0)\in\partial\{u>0\}\cap\Omega_T$ are such that $(x_k,t_k)\rightarrow(x_0,t_0)$, and $u_{\lambda_k}\rightarrow U$ uniformly on compact sets of $\R^{n+1}$ as $k\rightarrow\infty$, then, up to a subsequence in $\varepsilon$, one has $\varepsilon\lambda_k^{-1}\rightarrow0$, and
\begin{itemize}
\item[1.] $(u_\varepsilon)_{\lambda_k}\rightarrow U$ uniformly on compact sets of $\R^{n+1}$;
\item[2.] $\nabla(u_\varepsilon)_{\lambda_k}\rightarrow\nabla U$ in $L^2_{\loc}(\R^{n+1})$;
\item[3.] $\frac{\partial}{\partial t}(u_\varepsilon)_{\lambda_k}\rightarrow\frac{\partial}{\partial t}U$ weakly in $L^2_{\loc}(\R^{n+1})$;
\item[4.] $\nabla u_{\lambda_k}\rightarrow\nabla U$ in $L^2_{\loc}(\R^{n+1})$;
\item[5.] $\frac{\partial}{\partial t}u_{\lambda_k}\rightarrow\frac{\partial}{\partial t}U$ weakly in $L^2_{\loc}(\R^{n+1})$.
\end{itemize}
\end{lemma}
In order to understand general limits, first we need to analyze a particular limit of $u_\varepsilon$. For that purpose we will need the next lemma.
\begin{lemma}\label{l7.2}
Let $(A1)$, $(A4)$, \eqref{7.1} and \eqref{7.2} hold. If $u_\varepsilon$ is a solution of \eqref{Equation Pe} with $f_\varepsilon=0$ such that $u_\varepsilon\rightarrow u$ uniformly on compact subsets of $\Omega_T$ as $\varepsilon\rightarrow0$, then $B_\varepsilon(\tau):=\int_0^\tau\beta_\varepsilon(s)\,ds$ is precompact in $L^1(\Omega')$ for every $\Omega'\subset\subset\Omega_T$, and $B_\varepsilon(u_\varepsilon)$ converges in $L^1_{\loc}(\Omega_T)$ either to zero or to $L$, where $L$ is the total mass of $\beta$.
\end{lemma}
\begin{proof}
Note that since $F$ is assumed to be rotational invariant, then one has $a^\varepsilon_{ij}=a^\varepsilon_{ji}$. The proof of the lemma then follows by multiplying the equation \eqref{7.5} with a suitable function, namely with $\partial_ku_\varepsilon\Psi$, for $\Psi\in C_0^\infty(\Omega_T)$ and integrating over $\Omega_T$. Then the first (elliptic) term on the right hand side is handled exactly as in \cite[Lemma 6.6]{RT11} and the second (parabolic) term is handled as in \cite[Proposition 4.1 and Lemma 4.1]{W03}.
\end{proof}
\begin{definition}\label{d7.1}
A unit vector $\nu\in\mathbb{R}^{n}$ is called inward unit spacial normal to the free boundary $\partial\{u>0\}$ at $(x_0,t_0)\in\partial\{u>0\}$, for times $t\leq t_0$, in the parabolic measure theoretic sense, if
$$
\lim_{r\rightarrow0+}\frac{1}{r^{n+2}}\int_{Q_r^-(x_0,t_0)}|\chi_{\{u>0\}}-\chi_{\{(x,t);\,(x-x_0)\cdot\nu>0\}}|\,dx\,dt=0.
$$
We will refer to $(x_0,t_0)$ as a regular point.
\end{definition}
The previous lemma brings us closer to understanding a special limit.
\begin{lemma}\label{l7.3}
Let $(A1)$, $(A4)$, \eqref{7.1} and \eqref{7.2} hold. If $u_\varepsilon$ is a solution of \eqref{Equation Pe} with $f_\varepsilon=0$ such that $u_\varepsilon\rightarrow u:=\alpha[(x-x_0)\cdot\nu]^+$ uniformly on compact subsets of $\Omega_T$ as $\varepsilon\rightarrow0$, for $\alpha>0$ and a regular point $(x_0,t_0)\in\Omega_T\cap\partial\{u>0\}$, then
$$
\alpha=\sqrt{\frac{2L}{F^*(x_0,t_0,\nu\otimes\nu)}},
$$
where $F^*$ is defined by \eqref{7.3} and $L$ is the total mass of $\beta$.
\end{lemma}
\begin{proof}
Without loss of generality we may assume that $(x_0,t_0)=(0,0)$ and $\nu=e_1$. As before, note that since $F$ is rotational invariant, then the matrix $(a_{ij}^\varepsilon)$ is symmetric. As in the proof of \cite[Proposition 5.1]{LW06} (see also the proof of \cite[Proposition 6.7]{RT11}), we then multiply \eqref{7.5} by $\partial_1u_\varepsilon\Psi$, for a $\Psi\in C_0^\infty(\Omega_T)$, and integrate over $\Omega_T$. After integration by parts we arrive at
\begin{eqnarray}\label{7.7}
&&\int_{\Omega_T}\Big(\frac{1}{2}a_{ij}^\varepsilon\partial_iu_\varepsilon\partial_ju_\varepsilon\partial_1\Psi+\frac{1}{2}\partial_1a_{ij}^\varepsilon\partial_iu_\varepsilon\partial_ju_\varepsilon\Psi-a_{ij}^{\varepsilon}\partial_iu_\varepsilon\partial_1u_\varepsilon\partial_j\Psi\Big)\,dx\,dt\nonumber\\
&-&\int_{\Omega_T}\Big(\partial_ja_{ij}^\varepsilon\partial_iu_\varepsilon\partial_1u_\varepsilon\Psi+\partial_tu_\varepsilon\partial_1u_\varepsilon\Psi\Big)=-\int_{\Omega_T}B_\varepsilon(u_\varepsilon)\partial_1\Psi.
\end{eqnarray}
Using Theorem \ref{t5.1}, Lemma \ref{l7.1}, Lemma \ref{l7.2} and Proposition \ref{p7.1}, we pass to the limit in \eqref{7.7}, obtaining 
$$
-\frac{\alpha^2}{2}\int_{\{x_1=0\}}F^*(0,0,e_1\otimes e_1)\Psi=-(L-\tilde{L})\int_{\{x_1=0\}}\Psi,
$$
where $\tilde{L}$ is equal to either zero or $L$ (see Lemma \ref{l7.2}). Since $\Psi$ was an arbitrary smooth function, from the last identity we obtain that
$$
\alpha^2=\frac{2(L-\tilde{L})}{F^*(0,0,e_1\otimes e_1)},
$$
but since $\alpha$ is assumed to be positive, then $\tilde{L}=0$ by Lemma \ref{l7.2}, and the result follows.
\end{proof}
The following theorem is the main result of this section.
\begin{theorem}\label{t7.1}
Let $(A1)$, $(A4)$, \eqref{7.1} and \eqref{7.2} hold. Let also $u_\varepsilon$ be a solution of \eqref{Equation Pe} with $f_\varepsilon=0$ such that $u_\varepsilon\rightarrow u$ uniformly on compact subsets of $\Omega_T$ as $\varepsilon\rightarrow0$. If $(x_0,t_0)\in\partial\{u>0\}$ is a regular point, and $\{u=0\}$ has uniform positive density near $(x_0,t_0)$, then
$$
u(x,t)=\sqrt{\frac{2L}{F^*(x_0,t_0,\nu\otimes\nu)}}[(x-x_0)\cdot\nu]^++o(|x-x_0|+|t-t_0|^{1/2}),
$$
close to $(x_0,t_0)$, where $F^*$ is defined by \eqref{7.3} and $L$ is the total mass of $\beta$.
\end{theorem}
Our result is in accordance with the corresponding result in the elliptic case: as it is known from \cite{RT11}, near the free boundary points $u$ behaves in the following way
$$
u(x)=\sqrt{\frac{2L}{F^*(x_0,\nu\otimes\nu)}}[(x-x_0)\cdot\nu]^++o(|x-x_0|),
$$
where
$$
F^*(x,M):=\lim_{\varepsilon\rightarrow0}\varepsilon F(x,\varepsilon^{-1}M).
$$
Theorem \ref{t7.1} establishes the free boundary condition for general fully nonlinear parabolic problems, thus extending the corresponding results from \cite{Caf2,LO08}.

The proof of Theorem \ref{t7.1} is based on a Hopf type lemma and an upper bound for the gradient of the limiting solution near free boundary points.

In order to proceed, set
$$
\mathcal{B}_\tau:=Q_\tau^-(0,0)\cap\{x_1>0\}.
$$
\begin{lemma}\label{l7.4}
Let $u\in Lip(1,1/2)(\mathcal{B}_\tau)$ for some $\tau>0$. If $u\geq0$ satisfies
$$
F(x,t,D^2u)-\partial_tu=0\,\,\textrm{ in }\,\,\{u>0\},
$$
and $u\equiv0$ in $\{x_1=0\}$, then, in $\mathcal{B}_\tau$, $u$ has the asymptotic development
$$
u(x,t)=\alpha x_1^++o(|x|+|t|^{1/2}),
$$
for $\alpha\geq0$.
\end{lemma}
\begin{proof}
By Remark \ref{r7.1}, we have the needed regularity to follow the steps of the proof of the corresponding result from \cite{Caf2} (see also \cite{LO08,RT11}). It is based on the construction of a suitable barrier function. For details, we refer the reader to the proof of \cite[Corollary A.1]{Caf2} (see also the proof of \cite[Lemma 6.10]{RT11}).
\end{proof}
The next theorem provides an upper bound for the gradient of $u$ near free boundary points. Its proof is, step by step, a minor adaptation of the one of Theorem 6.1 of \cite{LO08} (see also the proof of \cite[Theorem 6.1]{Caf2} and \cite[Theorem 6.11]{RT11} for the elliptic case in the fully nonlinear setting).
\begin{theorem}\label{t7.2}
If the hypotheses of Theorem \ref{t7.1} are satisfied, then
$$
\limsup_{(x,t)\rightarrow(x_0,t_0)}|\nabla u(x,t)|\leq\sqrt{\frac{2L}{F^*(x_0,t_0,\nu\otimes\nu)}},
$$
where $F^*$ is given by \eqref{7.3} and $L$ is the total mass of $\beta$.
\end{theorem}
We are now ready to prove the main result of this section.
\begin{proof}[Proof of Theorem \ref{t7.1}]
We use the approach from \cite[Theorem 7.1]{LW06} (see also \cite[Theorem 7.1]{LO08}) with modifications for the fully nonlinear setting as carried out in the elliptic framework in \cite[Theorem 6.9]{RT11}. 

Observe that, without loss of generality, we may assume that $(x_0,t_0)=(0,0)$ and $\nu=e_1$. Set
$$
u_\lambda(x,t):=\frac{1}{\lambda}u(\lambda x,\lambda^2t),\,\,\lambda>0
$$
and let $r>0$ be small enough to guarantee $Q_r:=Q_r(0,0)\subset\subset\Omega_T$. Since $(0,0)$ is a free boundary (regular) point, then $u(0,0)=0$. Also $u_\lambda\in Lip(1,1/2)(Q_{r/\lambda})$. Hence there exist a subsequence $\lambda_k\rightarrow0$ and a function $U\in Lip(1,1/2)(\R^{n+1})$ such that $u_{\lambda_k}\rightarrow U$ uniformly on compact subsets of $\R^{n+1}$.

We aim to show that 
$$U=\sqrt{\frac{2L}{F^*(0,0,e_1\otimes e_1)}} \, x_1^+$$ 
for $t\leq 0$.
From the definition of the inward spacial normal in the parabolic measure theoretic sense we infer (see Definition \ref{d7.1}) that for every $d>0$, as $\lambda\rightarrow0$, one has
$$
|\{u_\lambda>0\}\cap\{x_1<0\}\cap Q_d^-(0,0)\}|\rightarrow0
$$
and
$$
|\{u_\lambda\equiv0\}\cap\{x_1>0\}\cap Q_d^-(0,0)\}|\rightarrow0,
$$
where $|E|$ is the $(n+1)$-dimensional Lebesgue measure of the set $E$. Therefore, $U\equiv0$ in $\{x_1<0\}\cap\{t\leq0\}$. On the other hand, $U\geq0$ and, due to Theorem \ref{t5.1}, solves \eqref{1.2} with $f\equiv0$. Also, $\{U>0\}\cap\{t<0\}\subset\{x_1>0\}$. By Lemma \ref{l7.4},
\begin{equation}\label{7.8}
U(x,t)=\alpha x_1^++o(|x|+|t|^{1/2})\,\,\textrm{ in }\,\,\{x_1>0\}\cap\{t<0\}.
\end{equation}
Since $\{u=0\}$ has uniform positive density near $(0,0)$, then $\alpha>0$ in \eqref{7.8}. 

Let $U_\lambda$ be the blow up of $U$ at $(0,0)$, i.e.,
$$
U_\lambda(x,t):=\frac{1}{\lambda}U(\lambda x,\lambda^2t).
$$
Then, for a sequence $\lambda_k\rightarrow0$, one has $U_{\lambda_k}\rightarrow\alpha x_1^+$ in $\{t\leq0\}$, uniformly in compact subsets. Also, from Lemma \ref{l7.1} and Lemma \ref{l7.4} (up to a subsequence), we have $u_{\lambda_k}\rightarrow\alpha x_1^+$ in $\{t\leq0\}$, uniformly in compact subsets. An application of Lemma \ref{l7.3} then gives
$$
\alpha=\sqrt{\frac{2L}{F^*(0,0,e_1\otimes e_1)}}.
$$
On the other hand, by Theorem \ref{t7.2}, one has
$$
|\nabla U|\leq\sqrt{\frac{2L}{F^*(0,0,e_1\otimes e_1)}}.
$$
Since $U\equiv0$ in $\{x_1=0\}\cap\{t\leq0\}$, we obtain
\begin{equation}\label{7.9}
U\leq\sqrt{\frac{2L}{F^*(0,0,e_1\otimes e_1)}}x_1\,\,\textrm{ in }\,\,\{x_1>0\}\cap\{t\leq0\}.
\end{equation}
Also $U$ is a subsolution of \eqref{1.2} (with $f\equiv0$) and satisfies \eqref{7.8} with $\alpha=\sqrt{\frac{2L}{F^*(0,0,e_1\otimes e_1)}}$. Applying Hopf's principle, we conclude that the equality holds in \eqref{7.9}. 
\end{proof}
\begin{remark}\label{r7.2}
When $f_\varepsilon\not\equiv0$ in \eqref{Equation Pe} but $(A3)$ holds, then the conclusion of Theorem \ref{t7.1} is still true. The proof is carried out in \cite[Theorem 7.1]{LW06} for the case of the Laplace operator. With minor adaptations, it works also in our framework.
\end{remark}

\bigskip

\noindent{\bf Acknowledgments.} This work was partially supported by CNPq-Brazil, FCT grant SFRH/BPD/92717/2013, and by CMUC -- UID/MAT/00324/2013, funded by the Portuguese government through FCT/MCTES and co-funded by the European Regional Development Fund through Partnership Agreement PT2020.

GCR thanks the Analysis group at CMUC for fostering a pleasant and productive scientific atmosphere during his postdoctoral program.

\end{document}